\documentclass[leqno]{amsart}
\usepackage{times}
\usepackage{amsfonts,amssymb,amsmath,amsgen,amsthm}
\usepackage{hyperref}
\usepackage{color}
\newcommand{\msc}[2][2000]{%
  \let\@oldtitle\@title%
  \gdef\@title{\@oldtitle\footnotetext{#1 \emph{Mathematics subject
        classification.} #2}}%
}

\theoremstyle{plain}
\newtheorem{theorem}{Theorem}[section]

\newtheorem{corollary}[theorem]{Corollary}
\newtheorem{proposition}[theorem]{Proposition}

\theoremstyle{remark}
\newtheorem{remark}[theorem]{Remark}

\def\C{{\mathbb C}}
\def\R{{\mathbb R}}
\def\O{\mathcal O}
\def\F{\mathcal F}

\def\bu{{\bf u}}

\def\({\left(}
\def\){\right)}
\def\<{\left\langle}
\def\>{\right\rangle}
\def\le{\leqslant}
\def\ge{\geqslant}

\def\Eq#1#2{\mathop{\sim}\limits_{#1\rightarrow#2}}
\def\Tend#1#2{\mathop{\longrightarrow}\limits_{#1\rightarrow#2}}

\def\d{{\partial}}
\def\eps{\varepsilon}
\def\om{\omega}
\def\si{{\sigma}}
\def\gg{\gamma}

\DeclareMathOperator{\RE}{Re}
\DeclareMathOperator{\IM}{Im}

\numberwithin{equation}{section}

\begin{document}

\title[Solutions to a singular Vlasov equation from a semiclassical
perspective]{Monokinetic solutions 
  to a singular Vlasov equation from a semiclassical perspective} 
\author[R. Carles]{R\'emi Carles}
\address{CNRS \& Univ. Montpellier\\Institut Montpelli\'erain
  Alexander Grothendieck
\\CC51\\F-34095 Montpellier}
\email{Remi.Carles@math.cnrs.fr}

\author[A. Nouri]{Anne Nouri}
\address{Aix-Marseille University, CNRS, Centrale Marseille\\
Institut de Math\'ematiques de Marseille, UMR 7373\\ F-13453 Marseille}
\email{anne.nouri@univ-amu.fr}

\begin{abstract}
  Solutions to a singular one-dimensional
  Vlasov equation are obtained as the semiclassical limit of the
  Wigner transform associated to a logarithmic Schr\"{o}dinger
  equation. Two frameworks are considered, regarding in particular the
  initial position density: Gaussian initial density, or smooth
  initial density away from vacuum. For Gaussian initial densities,
  the analysis also yields global solutions to the isothermal Euler
  system that do not enter the frame of regular solutions to
  hyperbolic systems by P.~D.~Lax.  
\end{abstract}
\thanks{RC was supported by the French ANR projects
  SchEq (ANR-12-JS01-0005-01) and BECASIM
  (ANR-12-MONU-0007-04). AN was supported by the French A*MIDEX project (Nr. ANR-11-IDEX-0001-02)} 
\maketitle

\section{Introduction and main results}
This paper is concerned by the Cauchy problem for the Vlasov equation 
\begin{equation}\label{vla1}
\partial _tf+\xi \partial _x f-\lambda \big( \partial _x\ln(\rho )\big) \partial _\xi f= 0,\quad f(0,x,\xi )= f_0(x,\xi ),\quad t>0, (x,\xi )\in \R ^2,
\end{equation}
where $\lambda \neq 0$ and $\rho (t,x)= \int _\R f(t,x,\xi )d\xi $. For $\lambda >0$ it arises in plasma physics, e.g. for quasineutral plasmas
in the core or tokamaks when one focuses on the direction of the magnetic lines. There, $f$ denotes the ionic distribution
function and the electrons of the plasma are assumed
adiabatic. \\
Due to the derivative of the density $\rho $ with respect to space in the force term $\partial _x\ln(\rho )$, this equation is highly singular. The Cauchy problem can in particular be proven to be well-posed for very specific initial data like mono-kinetic distribution functions of the form
\begin{eqnarray*}
f_0(x,\xi )=\rho _0(x)dx\otimes \delta _{\xi = v_0(x)},
\end{eqnarray*}
with time-dependent mono-kinetic solutions of the form 
\begin{eqnarray*}
f(t,x,\xi )=\rho (t,x)dx\otimes \delta _{\xi = v(t,x)}.
\end{eqnarray*}
From a fluid dynamics perspective for $\lambda >0$, it is well-known that $f(t,x,\xi ) = \rho (t,x)dx\otimes \delta _{\xi = v(t,x)}$ is a distributional solution of \eqref{vla1} if and only if its moments
\begin{eqnarray*}
\rho (t,x)= \int f(x,\xi )d\xi \quad \text{and}\quad \rho (t,x)v(t,x)= \int \xi f(t,x,\xi )d\xi 
\end{eqnarray*}
are solutions of the isothermal Euler system
\begin{equation}\label{isen1}
\left\{
  \begin{aligned}
    & \partial _t\rho +\partial _x(\rho v)= 0,\\
& \partial _t(\rho v)+\partial _x\( \rho v^2+\lambda \rho \)= 0.
  \end{aligned}
\right.
\end{equation}
This paper addresses the existence of mono-kinetic solutions to \eqref{vla1}, as limits of the Wigner transform of solutions to a logarithmic Schr\"{o}dinger equation introduced
in \cite{BiMy76} in the context of wave mechancis. As noticed there,
the evolution of initial Gaussian data can be computed rather
explicitly, a remark which yields our first main result:
\begin{theorem}\label{th1}
Let $\rho_*,\si_0>0$ and $\omega_0,p_0\in \R$. Set
\begin{equation*}
  \rho_0(x) =\rho_* e^{-\si_0x^2},\quad v_0(x) = \om_0 x+p_0,
\end{equation*}
and consider the ordinary differential equation
\begin{equation}\label{eq:theta}
  \ddot \gg = \frac{2\lambda \si_0}{\gg},\quad \gg(0)=1,\quad
  \dot\gg(0)=\omega_0. 
\end{equation}
Suppose that \eqref{eq:theta} has a solution $\gg\in C^2([0,T[ )$. Set
\begin{equation*}
  \rho(t,x) =\frac{\rho_*}{\gg(t)}
  e^{-\si_0(x-p_0t)^2/\gg(t)^2},\quad v(t,x) = \frac{\dot
    \gg(t)}{\gg(t)} x+p_0.
\end{equation*}
Then 
  \begin{equation*}
    \mu = \rho(t,x)dx\otimes \delta_{\xi=v(t,x)}
  \end{equation*}
is a measure solution to \eqref{vla1} on $[0,T[ $, with $\mu_{\mid
  t=0}=\rho_0dx \otimes \delta_{\xi=v_0}$.\\
In the case $\lambda>0$, the solution $\gamma$ is globally defined and
smooth, so $T$ can be taken arbitrarily large. In addition, it safisfies
\begin{equation*}
  \gg(t) \Eq t \infty  2  t \sqrt{\lambda \si_0\ln t} \hspace*{0.02in},
\quad\dot\gg(t)\Eq t \infty 2\sqrt{\lambda \si_0 \ln t} .
\end{equation*}
On the other hand, for $\lambda<0$, the solution $\gg$ becomes
singular in finite time. 
\end{theorem}

\begin{remark}
 In the case $\lambda >0$ the
global smooth solutions to the isothermal Euler system that are
obtained as part of the result, do not enter the frame of regular
solutions developed by Lax \cite{Lax73}. Indeed the Lax solutions
require bounded initial data, whereas the initial velocity in Theorem
\ref{th1} is unbounded.
\end{remark}
\begin{remark}
  To our knowledge, the case $\lambda<0$ does not correspond to a
  physical model related to \eqref{vla1}. However, the logarithmic
  nonlinear Schr\"odinger presented in Section~\ref{sec:nls-vlasov}
  was introduced initially exactly in the case $\lambda <0$
  (\cite{BiMy76}). We will see that the analysis of this case
  requires very little extra effort.  
\end{remark}
Our second result deals with more general solutions to \eqref{vla1} such that $\rho $ remains
 bounded away from zero. To do so, we do not consider $\rho\in
 H^s(\R)$ the standard Sobolev space, or $\rho \in
   \rho_*+H^s(\R)$ for 
 some $\rho_*>0$, but rather Zhidkov spaces, as
introduced in \cite{ZhidkovLNM,Zhidkov}, and further analyzed in
\cite{Gallo} in the case of Schr\"odinger equations. For $s\ge 1$, we
set
\begin{equation*}
  X^s(\R)= \left\{f\in L^\infty(\R),\quad f' \in H^{s-1}(\R)\right\}. 
\end{equation*}
 Note that being in dimension one, 
 $H^s(\R)\subset X^s(\R)$ holds for all $s\ge 1$, and $X^s(\R)$ is an algebra.\\
 The second result of the paper is the following theorem.

\begin{theorem}\label{th2}
Let $\lambda >0$ and $s\ge 2$. Suppose that $(\rho _0, \Phi _0)\in
X^s(\R )\times C(\R )$ with 
$\Phi_0'\in X^s(\R)$ and 
$\rho_0(x)\ge \rho _{0*}$ for some positive constant $\rho _{0*}$.\\ 
There are $T>0$, $\rho \in C([ 0,T[ ;
X^{s}(\R ))$, $\Phi\in C([0,T[ \times \R)$, with $
\d_x \Phi\in C([0,T[ ;X^s(\R)),$
such that  
\begin{eqnarray*}
\mu = \rho (t,x)dx\otimes \delta _{\xi = \partial _x\Phi (t,x)}\quad
  \text{with}\quad (\rho ,\Phi )\Big|_{t=0}= (\rho _0, \Phi _0), 
\end{eqnarray*}
is a measure solution to \eqref{vla1}. \\
Moreover, $\rho (t,\cdot )$ is bounded from below by a positive decreasing function of time.
\end{theorem}
\begin{remark}
   In \eqref{vla1}, the term $\partial _x\Big( \ln\int \mu d\xi
  \Big) \partial _\xi \mu $ can be considered in a weak sense as
  $\partial _\xi \Big( \partial _x\big( \ln\int \mu d\xi \big) \mu
  \Big) $, since $\int \mu d\xi = \rho \in C([ 0,T] ;C^1(\R ))$ and
  $\rho >0$.
\end{remark}
\begin{remark}
   Although the Cauchy problem for
the isothermal Euler \eqref{isen1} has 
global in time entropy weak solutions $(\rho ,u)\in L^\infty (\R _+\times \R
)^2$ (see \cite{Chen,Chen-LeFloch}), we cannot use them directly for our
purpose. Indeed, the 
momentum equation in  \eqref{isen1} is obtained from the kinetic
equation \eqref{vla1} by multiplying \eqref{vla1} by $\xi $ and
integrating the resulting equation with respect to $\xi $. This leads
to the product  
\begin{equation*}
\rho\hspace*{0.03in} \d _x \ln(\rho ).
\end{equation*}
Since it is not under a conservative form, it is well known that there
is no rigorous way to give a sense to this product for a general $\rho \in
L^\infty $. It is why we have recourse to regular solutions to the
isothermal Euler system, thus restricting for general initial bounded
data to local in time solutions far from vacuum. 
\end{remark}
\begin{remark}
  For $(\rho , v)$ with values in $] 0,+\infty [ \times \R $,
\eqref{isen1} is a strictly hyperbolic system. Consequently, for
$(\rho _0,v_0)\in (X^2(\R ))^2$ with $\rho _0\ge \alpha $ for some
$\alpha >0$, there are $T>0$ and \\
$(\rho , v)\in \big( C^1([ 0,T]
\times \R )\big) ^2$ solution to the Cauchy problem associated to
\eqref{isen1} and the initial datum $(\rho
_0,v_0)$. This theorem makes an extra connection with a logarithmic Schr\"odinger
equation, for which the above system corresponds to a limit system in
the semiclassical regime.
\end{remark}
\begin{remark}
  The assumption on the initial density
  $\rho_0$ is more general than merely $\rho_0\in \rho_{0*}+H^s(\R)$
  for some $s\ge 2$, since for instance, $\rho_0$ may have different
  limits as $x\to -\infty$ and $x\to +\infty$, or, even, no limit at
  all. 
\end{remark}
 
The plan of the paper is the following. Section~\ref{sec:2} recalls the main
steps of the derivation of the model for $\lambda >0$ and related mathematical
results. In Section~\ref{sec:nls-vlasov}, we show how \eqref{vla1} can
be obtained formally from a nonlinear Schr\"odinger equation through
the semiclassical limit. Section~\ref{sec:gaussian} establishes
Theorem~\ref{th1}. In Section~\ref{sec:bkw} , we prove Theorem
\ref{th2}.  Section~\ref{sec:6} adapts the proof of the boundedness
from below of the density in the isotropic case by  \cite{Ch-p} to the
isothermal case.\\  
\hspace*{0.1in}\\
\hspace*{0.1in}\\
\section{Derivation of the model and related results}
\label{sec:2}
In this section, we recall the main lines of the derivation of the model with $\lambda >0$, used for studying fusion plasmas (\cite{Ghendrih-Hauray-Nouri}). The evolution of the ions in the core of such plasmas is well described by the Vlasov equation
\begin{eqnarray*}
\d _tf+v\cdot \nabla _xf+\frac{Ze}{m_i}(-\nabla _x\Phi +v\wedge B)\cdot \nabla _vf= 0,
\end{eqnarray*}
where $f$ is the ionic distribution function depending on time, position (in the domain $\Omega $ of the plasma) and velocity (in $\R ^3$), $Ze$ and $m_i$ are the ion charge and mass respectively. The electric potential $\Phi $ and the magnetic field $B$ should be governed by the Maxwell equations. But a finite Larmor radius approximation is derived in the limit of a large and uniform external magnetic field. This leads to the following equation for the ionic distribution function $\overline{f}$ in gyro coordinates,
\begin{equation}\label{gyro}
\d _t\overline{f}+v_\parallel \d _{x_\parallel }\overline{f}-J_{\rho
  _L}^0(\d _{x_\parallel }\Phi )\d _{v_\parallel }\overline{f}-\big(
J_{\rho _L}^0\nabla _{x_\perp }\Phi \big) ^\perp \cdot \nabla
_{x_\perp }\overline{f}= 0. 
\end{equation}
Here, the index $\parallel $ (resp. $\perp $) refers to the direction parallel (resp. orthogonal) to the external magnetic field. For any vector $u= (u_i)\in \R ^3$, $u^\perp $ denotes the vector $(u_2,-u_1,0)$. The operator $J_{\rho _L}^0$ is a Bessel operator performing averages on circles of Larmor radius $\rho _L$ in planes orthogonal to the magnetic field. Since it is not used in this paper, we do not enter into more details about it. The electrons move quite more quickly than the ions, so that their density $n_e$ is given in terms of the electric potential $\Phi $ by the Maxwell-Boltzmann equation
\begin{equation}\label{max-boltz}
n_e= n_0e^{\frac{e}{T_e}(\Phi -<\Phi >)}.
\end{equation}
Here, $e$ (resp. $T_e$) is the electronic charge (resp. temperature),
and $<\Phi >$ is the average of the potential on a magnetic field
line. Due to the electroneutrality of the plasma, the Poisson equation
is replaced by the electroneutrality equation 
\begin{eqnarray*}
n_e= \rho ,
\end{eqnarray*}
where $\rho $ is the ionic density. The operator $J_{\rho _L}^0$
induces some regularity in the orthogonal direction, but none in the
parallel direction. The two-dimensional dynamics in the direction
perpendicular to the magnetic field is studied in
\cite{Hauray-Nouri}. In order to analyze the difficulty coming from
the highly singular term $J_{\rho _L}^0(\d _{x_\parallel }\Phi )\d
_{v_\parallel }\overline{f}$, we restrict to a one-dimensional spatial
setting, e.g. by considering ionic distribution functions written in
the form 
\begin{eqnarray*}
f(t,x,v)= f_\parallel (t,x_\parallel ,v_\parallel )f_\perp (\lvert v_\perp \rvert ),
\end{eqnarray*}
with 
\begin{eqnarray*}
\int _0^{+\infty }f_\perp (\lvert v_\perp \rvert )2\pi \lvert v_\perp
  \rvert d\lvert v_\perp \rvert  = 1.
\end{eqnarray*}
Then the term $f_\perp $ has no incidence in equation \eqref{gyro} and
can be factorized. The equation that $f_\parallel $ should solve is 
\begin{eqnarray*}
\partial _tf_\parallel +v_\parallel \partial _{x_\parallel} f_\parallel -\lambda \big( \partial _{x_\parallel }\ln(\rho _\parallel )\big) \partial _{v_\parallel } f_\parallel = 0,\quad t>0, (x_\parallel ,v_\parallel )\in \R ^2,
\end{eqnarray*}
i.e. the partial differential equation in \eqref{vla1} for
$f_\parallel $ (resp. $\rho _\parallel $, $x_\parallel $, $v_\parallel
$) denoted by $f$ (resp. $\rho $, $x$, $\xi $),  and $\lambda =
\frac{T_e}{e}$.\\ 
\hspace*{0.1in}\\
Mathematical results related to \eqref{vla1} have been obtained for a
system close to equilibrium, i.e. in the case where the departure of
the electric potential $\Phi $ from its average along the magnetic
lines $<\Phi >$ is small. Equation \eqref{max-boltz} simplifies into 
\begin{eqnarray*}
n_e= n_0\( 1+\frac{e}{T_e}\(\Phi -<\Phi >\)\) ,
\end{eqnarray*}
so that (\ref{vla1}) is replaced by
\begin{equation}\label{vdpe}
\partial _tf+\xi \partial _x f-\lambda \big( \partial _x\rho
\big) \partial _\xi f= 0,\quad t>0, (x,\xi )\in \R ^2. 
\end{equation}
The Cauchy problem for \eqref{vdpe} is locally well-posed either for initial
analytic data \cite{Jabin-Nouri} or in Sobolev spaces and satisfying a Penrose stability condition \cite{HKR-p}, but is ill-posed in the sense of
Hadamard for regular initial data in Sobolev spaces and arbitrarily
small time \cite{Bardos-Nouri}.

\section{From Schr\"odinger  to Vlasov  via Euler}
\label{sec:nls-vlasov}

In this section, we show how a logarithmic Schr\"odinger
equation can be formally related to  the isothermal Euler system \eqref{isen1}, in 
the semiclassical limit.
For $\eps >0$, consider the Cauchy problem for the Schr\"{o}dinger equation
\begin{equation}\label{Schrod1}
i\eps \partial _tu^\eps +\frac{\eps ^2}{2}\partial ^2_{xx}u^\eps =
\lambda \hspace*{0.02in}\ln(\lvert u^\eps \rvert ^2)u^\eps ,\quad
u^\eps (0,x)= \sqrt{\rho_0(x)}e^{i\Phi _0(x)/\eps }. 
\end{equation}
Following the idea from \cite{Grenier98}, any function $u^\eps =
a^\eps e^{i\Phi ^\eps/\eps }$,  
with $(t,x)\mapsto a^\eps (t,x)\in \mathbb{C}$ and $(t,x)\mapsto \Phi
^\eps(t,x) \in \R $ solutions to the quasilinear problem 
\begin{equation}\label{Phi-eps}
\partial _t\Phi ^\eps +\frac{(\partial _x\Phi ^\eps )^2}{2}+\lambda \hspace*{0.02in}\ln\big( \lvert a^\eps \rvert ^2\big) = 0,\quad \Phi ^\eps (0,x )= \Phi _0(x ),
\end{equation}
\begin{equation}\label{a-eps}
\partial _ta^\eps +\partial _x\Phi ^\eps \partial _xa^\eps
+\frac{a^\eps }{2}\partial ^2_{xx}\Phi ^\eps = i\frac{\eps
}{2}\partial ^2_{xx}a^\eps ,\quad a^\eps (0,x)=
\sqrt{\rho_0(x)}=:a_0(x ), 
\end{equation} 
is a solution to \eqref{Schrod1}. An important remark is that by
allowing $a^\eps$ to be complex-valued (even though its initial datum
is real-valued), one gains a degree of freedom
to dispatch terms from \eqref{Schrod1} into
\eqref{Phi-eps}-\eqref{a-eps}, and the choice introduced by Grenier is
much more robust than the Madelung transform when semiclassical limit is
considered (see \cite{CaDaSa12}).
\smallbreak

Determining $\Phi ^\eps $
solution to \eqref{Phi-eps} turns out to be equivalent to determining $v^\eps
= \partial _x\Phi ^\eps $ and $a^\eps$ solution to 
\begin{equation}\label{av-eps}
  \left\{
\begin{aligned}
  &\partial _tv^\eps +v^\eps \partial _xv^\eps +\lambda
\hspace*{0.02in}\partial _x\hspace*{0.02in}\ln\( \lvert a^\eps
\rvert ^2\)= 0,\quad v^\eps (0,x )= \Phi _0^\prime (x ),\\
&\d_t a^\eps + v^\eps \d_x a^\eps +\frac{a^\eps}{2}\d_x v^\eps =
i\frac{\eps}{2}\d_{xx}^2 a^\eps,\quad a^\eps(0,x)=a_0(x). 
  \end{aligned}
\right.
\end{equation}
Indeed, given $(v^\eps,a^\eps)$ solution to \eqref{av-eps}, we can define
$\Phi^\eps$ by
\begin{equation}\label{eq:v2phi}
  \Phi^\eps(t,x) = \Phi_0(x) -\int_0^t \(\frac{1}{2}|v^\eps(\tau,x)|^2
  +\lambda \ln \(|a^\eps(\tau,x)|^2\)\)d\tau.
\end{equation}
We check
that
\begin{equation*}
  \d_t\(\d_x \Phi^\eps - v^\eps\) = \d_x\d_t \Phi^\eps -\d_t v^\eps =0,
\end{equation*}
so that $v^\eps=\d_x \Phi^\eps$ and $\Phi ^\eps $
solves \eqref{Phi-eps}, and $a^\eps$ solves \eqref{a-eps}.
\smallbreak

Passing formally to the limit $\eps\to 0$ in \eqref{av-eps}, we get
the system
\begin{equation}
  \label{eq:limit}
  \left\{
    \begin{aligned}
      &\d_t v+v\d_x v + \lambda \d_x\ln(|a|^2)=0,\quad v_{\mid
        t=0}=\Phi_0^\prime,\\
& \d_t a +v\d_x a +\frac{a}{2}\d_x v =0,\quad a_{\mid
  t=0}=\sqrt{\rho_0},
    \end{aligned}
\right.
\end{equation} 
which turns out to be the symmetrized version of \eqref{isen1}, with
$\rho=|a|^2$ (see \cite{JYC90,MUK86}). As noticed in the introduction,
we then formally obtain a solution to \eqref{vla1} by setting
\begin{equation*}
  f(t,x,\xi) = |a(t,x)|^2dx\otimes \delta_{\xi=v(t,x)}.
\end{equation*}
A more direct link from \eqref{Schrod1} to \eqref{vla1} is provided by
the notion of Wigner measure. The Wigner transform of $u^\eps $,
solution to \eqref{schrod}--\eqref{init1}, is 
defined by (see e.g. \cite{BurqMesures,GMMP,LionsPaul})
\begin{equation}\label{df1-wig}
W^\eps (t,x,\xi )= \int_{\R}e^{iy\xi}u^\eps \(t,x-\frac{\eps
}{2}y\)\overline{u^\eps }\(t,x+\frac{\eps }{2}y\)dy,\quad (t,x,\xi
)\in [ 0,T] \times \R ^2.  
\end{equation}
Up to the extraction of a subsequence, $W^\eps$ converges to a
non-negative measure on the phase space. In general, several limits
may exist (see the above references). We will see that in the
framework of this paper, the limit is unique, and solves
\eqref{vla1}.

\section{Gaussian initial data}
\label{sec:gaussian}
Let us first study the ordinary differential equation
\eqref{eq:theta}. Local existence and 
uniqueness of a $C^1$ solution stem from Cauchy-Lipschitz Theorem.
Indeed, the nonlinearity in the above equation is locally Lipschitzean
away from $\{\gg =0\}$. We
now address the global existence issue. 
\smallbreak

In the case $\lambda>0$, multiplying \eqref{eq:theta} by $\dot \gg$ and
integrating yields
\begin{equation}\label{energy1}
  (\dot \gg)^2 = \omega_0^2 +4\lambda \si_0\ln \gg. 
\end{equation}
This readily shows that $\gg$ is bounded from below away from zero,
so the flow is global. \\
In the case $\lambda<0$, suppose that $\gg$ is bounded away from
zero, i.e. there is $\delta >0$ such that $\gg(t)\ge \delta $. Then \eqref{eq:theta} yields
\begin{equation*}
  \ddot\gg\le \frac{2\lambda \si_0}{\delta},
\end{equation*}
hence $\gg(t)\le \frac{\lambda \si_0}{\delta}t^2 +\omega_0t+1$, and a
contradiction for $t$ sufficiently large. Now suppose that $\gg\in
C^2(0,\infty)$ with $\gg>0$: from the above argument, there exists
a sequence $t_n$ along which $\gg(t_n)\to 0^+$. From
\eqref{eq:theta}, $\ddot \gg(t_n)\to -\infty$, hence a
contradiction.  
\smallbreak
To prove Theorem~\ref{th1}, we start from a semi-classically scaled logarithmic
Schr\"odinger equation,
\begin{equation}\label{schrod}
i\eps \partial _tu^\eps +\frac{\eps ^2}{2}\partial ^2_{xx}u^\eps = \lambda \hspace*{0.02in}\ln (|u^\eps |^2)u^\eps ,
\end{equation}
together with the initial value
\begin{equation}\label{init1}
  u^\eps(0,x) =\sqrt{\rho_*}e^{-\si_0 x^2/2} e^{i\omega_0x^2/(2\eps)} e^{ip_0x/\eps}.
\end{equation}
Such an initial datum does not fit into the framework of
\cite{ACIHP}, since it goes to zero at infinity. However, as noted in
\cite{BiMy76}, for
\emph{fixed} $\eps>0$, the solution to \eqref{schrod} with such an
initial datum can be computed rather explicitly. Indeed, bearing in
mind the
propagation of coherent states in the semi-classical regime
(\cite{Hag80}, see also \cite{CoRoBook})  it is
consistent to look for a solution of the form
\begin{equation}\label{eq:ugauss}
  u^\eps(t,x) = b^\eps(t)e^{-\Omega^\eps(t)(x-q(t))^2/2 +ip(t)(x-q(t))/\eps+iS(t)/\eps},
\end{equation}
with $b^\eps,\Omega^\eps\in \C$, and $q,p,S\in \R$. Plugging this ansatz
into \eqref{Schrod1}, we find:
\begin{align*}
  &i\eps \dot b^\eps -i \eps \dot \Omega^\eps \frac{(x-q)^2}{2}b^\eps
    +i\eps \Omega ^\eps \dot q (x-q)b^\eps-\dot p 
  (x-q)b^\eps +p\dot q b^\eps -\frac{p^2}{2} b^\eps-\dot S b^\eps -\eps^2
  \frac{\Omega^\eps}{2}b^\eps\\
 & +\eps^2(\Omega^\eps)^2\frac{(x-q)^2}{2}b^\eps-ip\eps \Omega^\eps (x-q)b^\eps =
  \lambda b^\eps\ln |b^\eps|^2 -\lambda b^\eps (x-q)^2\RE \Omega^\eps. 
\end{align*}
Cancelling the polynomial in $x-q$, sufficient conditions for the
previous equation to hold are 
\begin{align}
  \label{eq:hamil}
&  \dot q=p,\quad \dot p=0,\quad \dot S = p\dot q
  -\frac{p^2}{2}=\frac{p^2}{2},\\
\label{eq:b}
&i\eps \dot b^\eps = \frac{\eps^2}{2}\Omega^\eps b^\eps +\lambda b
  ^\eps\ln |b^\eps|^2,\\ 
\label{eq:Omega}
& i\eps \dot \Omega^\eps = \eps^2(\Omega^\eps)^2 +2\lambda \RE
  \Omega^\eps. 
\end{align}
Equation~\eqref{eq:hamil} corresponds to the classical Hamiltonian
flow in the absence of external force, and the associated classical
action. We compute exactly
\begin{equation*}
  p(t) = p_0,\quad q(t) = p_0 t,\quad S(t) = \frac{p_0^2}{2}t.
\end{equation*}
Since we are eventually interested only in the modulus of $b^\eps$, we
infer from \eqref{eq:b}:
\begin{equation*}
  \frac{d}{dt}|b^\eps|^2=2\RE \bar b^\eps \dot b^\eps = \eps
  |b^\eps|^2\IM \Omega^\eps,
\end{equation*}
hence
\begin{equation}\label{modulus-b_eps}
  |b^\eps(t)|^2 =  \rho _*e^{\eps \int_0^t \IM \Omega^\eps(s)ds}.
\end{equation}
For
fixed $\eps>0$ and $\RE\Omega^\eps(0)>0$,  \eqref{eq:Omega} has a
unique, global solution 
$\Omega^\eps \in C(\R)$, whose large time behavior depends on the sign
of $\lambda$. Indeed, in a similar fashion as in
\cite{LiWa06}, we seek $\Omega^\eps$ of the form
\begin{equation*}
  \Omega^\eps = -i\frac{\dot \eta^\eps}{\eps\eta^\eps}.
\end{equation*}
Then \eqref{eq:Omega} becomes
\begin{equation*}
  \ddot \eta^\eps = \frac{2\lambda}{\eps}\eta^\eps\IM\frac{\dot
      \eta^\eps}{\eta^\eps}.  
\end{equation*}
Introducing the scaled polar decomposition $\eta^\eps=\gg^\eps
e^{i\eps\alpha^\eps}$, $\Omega^\eps$ is given by
\begin{equation}\label{eq:Omegaralpha}
  \Omega^\eps = \dot \alpha^\eps -\frac{i}{\eps}\frac{\dot \gg^\eps}{\gg^\eps},
\end{equation}
and the 
above equation reads:
\begin{equation*}
   \ddot \gg^\eps - \eps^2 \gg^\eps(\dot \alpha^\eps)^2 =
    2\lambda \gg^\eps \dot \alpha^\eps,\qquad
\ddot \alpha^\eps \gg^\eps +2\dot \alpha^\eps \dot \gg^\eps=0.
 \end{equation*}
The second equation yields
\begin{equation*}
 \frac{d}{dt}\( (\gg^\eps)^2\dot \alpha^\eps\) = 0.
\end{equation*}
In view of \eqref{eq:Omegaralpha}, we have a degree of freedom to set
$\gg^\eps(0)$. Setting $\gg^\eps(0)=1$, we have
\begin{equation*}
  \dot \alpha^\eps(0)= \RE\Omega^\eps(0)=\si_0,\quad \dot
  \gg^\eps(0)=-\eps\IM \Omega^\eps(0)= \omega_0. 
\end{equation*}
Therefore, we have
\begin{equation}\label{eq:rdotalpha}
  (\gg^\eps)^2\dot \alpha^\eps = \si_0.
\end{equation}
The  equation for $\gg^\eps$ boils down to 
\begin{equation*}
  \ddot \gg^\eps=  \frac{\eps^2\si_0^2}{(\gg^\eps)^3} +
  \frac{2\lambda \si_0}{\gg^\eps} , 
\end{equation*}
along with the initial data that we recall:
\begin{equation*}
  \gg^\eps(0)=1,\quad \dot  \gg^\eps(0)=\omega_0. 
\end{equation*}
 Local existence and
uniqueness of a $C^1$ solution stem from Cauchy-Lipschitz Theorem.
Indeed, the nonlinearity in the above equation is locally Lipschitzean
away from $\{\gg^\eps=0\}$. We
now address the global existence issue. 
Multiplying the above equation by $\dot \gg^\eps$ and integrating in
time, we infer:
\begin{equation*}
  (\dot \gg^\eps)^2 = 4\lambda \si_0 \ln \gg^\eps
  -\frac{\eps^2\si_0^2}{(\gg^\eps)^2}+ \eps^2\si_0^2+\omega_0^2. 
\end{equation*}
This shows that for fixed $\eps>0$, $\gg^\eps$ remains bounded away from
zero, for if we had $\gg^\eps(t_n)\to 0$ for some sequence $t_n$, then
the above right hand side would become negative for $n$ large enough,
hence a contradiction. Thus, for fixed $\eps>0$,  \eqref{eq:Omega} has
a unique, global solution $\Omega^\eps \in C^1(\R)$. \\
In view of \eqref{eq:rdotalpha}, \eqref{eq:Omegaralpha} also reads
\begin{equation}\label{Omega_eps}
  \Omega^\eps = \frac{\si_0}{(\gg^\eps)^2}-\frac{i}{\eps}\frac{\dot \gg^\eps}{\gg^\eps}.
\end{equation}
\smallbreak

Given \eqref{modulus-b_eps} and \eqref{Omega_eps}, the Wigner
transform \eqref{df1-wig} of $u^\eps$ is given by
\begin{equation}\label{df2-wig}
W^\eps (t,x,\xi )= \frac{\rho _*}{\gamma ^{\eps }(t)}e^{-\frac{\sigma _0(x-p_0t)^2}{(\gamma ^\eps (t))^2}}\int e^{-\frac{\sigma _0\eps ^2y}{4(\gamma ^\eps (t))^2}}e^{iy(\xi-\frac{\dot \gg^\eps}{\gg^\eps}(x-p_0t)-p_0)}dy.
\end{equation}
On every time interval such that $\gamma $ is bounded away from zero, the Gronwall lemma shows that
\begin{eqnarray*}
\gamma ^\eps -\gamma = \O(\eps),\quad \dot\gg^\eps-\dot\gg=\O(\eps).
\end{eqnarray*}
Consequently, when $\eps \to 0$, the Wigner transform $W^\eps $ of $u^\eps $ weakly
converges  to the bounded measure  
\begin{equation}\label{df-mu}
\mu (t,dx,d\xi)= \frac{\rho _*}{\gamma (t)}e^{-\frac{\sigma _0(x-p_0t)^2}{(\gamma (t))^2}}dx\otimes \delta _{\xi= \frac{\dot \gg}{\gg}(x-p_0t)+p_0}.
\end{equation}
Straightforward computations show that
\begin{eqnarray*}
(\rho (t,x), v(t,x)):= \Big( \frac{\rho _*}{\gamma (t)}e^{-\frac{\sigma _0(x-p_0t)^2}{(\gamma (t))^2}}, \frac{\dot \gg (t)}{\gg (t)}(x-p_0t)+p_0\Big) 
\end{eqnarray*}
is a solution to the isothermal Euler system (\ref{isen1}) on the time interval of the existence of $\gamma $.
\smallbreak
To conclude, we consider the large time behavior of $\gg$ for $\lambda>0$. If $\gg$ was bounded from above, $\gg\le
M$, then \eqref{eq:theta} would yield
\begin{equation*}
  \ddot \gg \ge \frac{2\lambda \si_0}{M}>0,
\end{equation*}
hence a contradiction. Therefore, for $\underline t$ sufficiently
large, $\gg(\underline t)\ge 1$ and $\dot \gg(\underline
t)>0$. Since $\ddot \gg\ge 0$, integration then shows
\begin{equation*}
  \gg(t)\ge \dot \gg(\underline t)(t-\underline t) +1\Tend t
  \infty \infty,
\end{equation*}
and $\dot \gg>0$ for $t$ sufficiently large. 
The asymptotic behavior announced in Theorem~\ref{th2} then follows by
integrating the identity
\begin{equation*}
  \frac{d\gg}{\sqrt{\omega_0^2 +4\lambda \si_0\ln \gg}}=dt,
\end{equation*}
to obtain the asymptotic behavior of $\gg$. 
\begin{equation*}
  t = \int \frac{d\gg}{\sqrt{\omega_0^2 +4\lambda \si_0\ln \gg}} =
  \frac{1}{2\lambda\si_0}\int e^{(y^2-\omega_0^2)/(4\lambda\si_0)}dy, 
\end{equation*}
where we have changed the variable as
\begin{equation*}
  y = \sqrt{4\lambda  \si_0\ln \gg +\omega_0^2}.
\end{equation*}
Recall that the Dawson function, defined by
\begin{equation*}
  F(x) =e^{-x^2}\int_0^x e^{y^2}dy
\end{equation*}
satisfies
\begin{equation*}
  F(x)\Eq x {+\infty} \frac{1}{2x}, 
\end{equation*}
(see e.g. \cite{AbSt64}), we infer that
\begin{equation*}
  \gg(t) \Eq t {+\infty} 2t\sqrt{\lambda\si_0\ln t}.
\end{equation*}
Since
\begin{equation*}
  \dot \gg = \sqrt{\omega_0^2 +4\lambda \si_0\ln \gg},
\end{equation*}
we conclude
\begin{equation*}
   \dot \gg(t)\Eq t {+\infty} 2\sqrt{\lambda \si_0\ln t}.
\end{equation*}
In the case $\lambda<0$, suppose that $\gg$ is bounded away from
zero, $\gg(t)\ge \delta>0$. Then \eqref{eq:theta} yields
\begin{equation*}
  \ddot\gg\le \frac{2\lambda \si_0}{\delta},
\end{equation*}
hence $\gg(t)\le \frac{\lambda \si_0}{\delta}t^2 +\omega_0t+1$, and a
contradiction for $t$ sufficiently large. Now suppose that $\gg\in
C^2(0,\infty)$ with $\gg>0$: from the above argument, there exists
a sequence $t_n$ along which $\gg(t_n)\to 0^+$. From
\eqref{eq:theta}, $\ddot \gg(t_n)\to -\infty$, hence a
contradiction.  
\smallbreak

This ends the proof of Theorem \ref{th1}.

\section{WKB analysis}
\label{sec:bkw}

In this section, we justify the formal approach presented in
Section~\ref{sec:nls-vlasov} in the framework of Theorem~\ref{th2},
thus proving this result.

\subsection{Constructing the solution}
\label{sec:constr-solut}

Note that for fixed $\eps>0$, the Cauchy problem for \eqref{Schrod1}
has been considered in \cite{CaHa80} (see also
\cite[Section~9.1]{CazCourant}), for initial data in the class
\begin{equation*}
  W=\left\{ f\in H^1(\R),\quad \int_\R |f(x)|^2 \left| \ln
      |f(x)|^2\right|dx<\infty\right\}. 
\end{equation*}
This class is not compatible with the assumption $\rho(x)\ge
\rho_{0*}>0$ from Theorem~\ref{th1}, which is equivalent to 
$|u^\eps(0,x)|^2=|a_0(x)|^2\ge \rho_{0*}>0$ in the approach that we
follow. Therefore, we choose to rather work in Zhidkov spaces
$X^s(\R)$.
The system \eqref{av-eps} has a unique smooth solution as
stated in the following proposition, which includes the case
$\eps=0$. 
\begin{proposition}\label{prop:existence}
  Let $s>5/2$ and $\lambda>0$. Suppose that $\rho_0,\Phi_0^\prime\in
  X^s(\R)$, with 
  \begin{eqnarray*}
\rho_0(x)\ge \rho_{0*}>0.
\end{eqnarray*}
Then there 
  exists $T$ independent of $s>5/2$ and $\eps\in [0,1]$, and a
  unique solution \\
  $(a^\eps,v^\eps)\in C([0,T];X^s\times X^s)$ to
  \eqref{av-eps}. 
\end{proposition} 
\begin{proof}
  This result is a rather direct consequence of
  \cite[Proposition~2.1]{ACIHP}, whose proof we recall the main
  idea. Separate real and imaginary parts of $a^\eps$,
  $a^\eps=a_1^\eps+ia_2^\eps$, and introduce
\begin{equation*}
  \bu^\eps = 
    \begin{pmatrix}
       a_1^\eps \\
       a_2^\eps \\
       v^\eps
    \end{pmatrix}
 , \quad
\bu_0 = 
    \begin{pmatrix}
       \sqrt{\rho_0}\\
      0 \\
       \Phi_0^\prime 
    \end{pmatrix}
\ , \quad
L = 
  \begin{pmatrix}
   0  &-\d_{xx}^2 &0   \\
   \d_{xx}^2  & 0 &0  \\
   0& 0 &0 \\
   \end{pmatrix}
\end{equation*}
\begin{equation*}
\text{and}\quad A(\bu)
= \begin{pmatrix}
      v  & 0& \frac{a_1 }{2} \\
     0 &  v & \frac{a_2}{2} \\
     \frac{2\lambda a_1}{a_1^2+a_2^2}
     &\frac{2\lambda a_2}{a_1^2+a_2^2} &  v
    \end{pmatrix}.
\end{equation*}
We now have the system:
\begin{equation}
  \label{eq:systhypgeneral}
\partial_t \bu^\eps +
  A(\bu^\eps)\partial_x \bu^\eps = \frac{\eps}{2} L
  \bu^\eps\quad ;\quad
 \bu^\eps_{\mid t=0}=\bu_0.
\end{equation}
Since $\rho_0$ is bounded away from zero, its square root is also in
$X^s$, so that $\bu_0\in X^s(\R)^3$. 
The matrix $A$ is symmetrized by the matrix
\begin{equation*}
  S=\left(
    \begin{array}[l]{cc}
     I_2 & 0\\
     0& \frac{a_1^2+a_2^2}{4\lambda}
    \end{array}
\right),
\end{equation*}
which is symmetric positive if and only if $a_1^2+a_2^2>0$, that is,
so long as no vacuum appears. By assumption, 
\begin{equation*}
  (a_1^\eps)^2+(a_2^\eps)^2\Big|_{t=0}\ge \rho_{0*}>0.
\end{equation*}
Then the main idea is that the operator $L$ is skew-symmetric, and so
does not appear in $L^2$-based energy estimates. Standard tame estimates (see
e.g. \cite{Majda,Taylor3}) do not involve the $L^2$ norm of
$\bu^\eps$, and so the only aspect remaining is that
$L^\infty$-estimates can be obtained rather directly. So long as, say, 
\begin{equation}
  \label{eq:solong}
  (a_1^\eps(t,x))^2+(a_2^\eps(t,x))^2 \ge \frac{\rho_{0*}}{2},\quad \forall
  x\in \R,
\end{equation}
we have:
\begin{align*}
  \|\bu^\eps(t)\|_{L^\infty} &\le \|\bu_0\|_{L^\infty} + \int_0^t
  \|A(\bu^\eps(\tau))\d_x \bu^\eps(\tau)\|_{L^\infty}d\tau + \int_0^t\|\d_{xx}^2
                               a^\eps(\tau)\|_{L^\infty}d\tau\\
&\le \|\bu_0\|_{L^\infty} + C\int_0^t
  \|\bu^\eps(\tau)\|_{L^\infty}\|\d_x \bu^\eps(\tau)\|_{L^\infty}d\tau
  + \int_0^t\|\d_{xx}^2 
                               a^\eps(\tau)\|_{L^\infty}d\tau\\
&\le \|\bu_0\|_{L^\infty} + C\int_0^t
  \|\bu^\eps(\tau)\|_{X^s}^2d\tau + \int_0^t\|\bu^\eps(\tau)\|_{X^s}d\tau,
\end{align*}
where we have used Sobolev embedding under the assumption
$s>5/2$. Indeed, by definition, $X^s\subset
  L^\infty$, and for $\bu^\eps\in X^s$, $\d_x \bu^\eps\in
  H^{s-1}\subset L^\infty$, provided that $s>3/2$, and similarly, $\d_{xx}^2
  \bu^\eps\in L^\infty$ for $s>5/2$. \\ 
Now we set $P=(I-\d_{xx}^2)^{(s-1)/2}\d_x$, so that $\|f\|_{X^s}\approx
\|f\|_{L^\infty}+\|Pf\|_{L^2}$, and denote by
\begin{equation*}
  \<f,g\>=\int_{-\infty}^\infty f(x)\overline{g(x)}dx,
\end{equation*}
the scalar product in $L^2$. Since $L$ is skew-symmetric and  $S$ is real-valued,
\[ \begin{aligned}
  &\frac{d}{dt}\<SP\bu^\eps(t),P\bu^\eps(t)\>\\
  &= \<(\d_t S)
  P\bu^\eps(t),P\bu^\eps(t)\>+2\RE\<S\d_t
  P\bu^\eps(t),P\bu^\eps(t)\> \\
 &= \< (\d_t S)
  P\bu^\eps(t),P\bu^\eps(t)\>+\eps \RE\<SLP\bu^\eps(t),P\bu^\eps(t)\>
  \\
&\quad -2\RE\<SP\Big( A(u^\eps (t))\d _xu^\eps (t)\Big) ,P\bu^\eps(t)\> .
\end{aligned}\]
So long as \eqref{eq:solong} holds,
we have the following set of estimates. First,
\begin{align*}
  \< (\d_t S )P\bu^\eps(t),P\bu^\eps(t)\>&\le \|\d_t
  S\|_{L^\infty}\|P\bu^\eps(t)\|^2_{L^2}\\
&\le C\(
  \|\bu^\eps(t)\|_{L^\infty}\)\|\d_t \bu^\eps(t)\|_{L^\infty}
  \|\bu^\eps(t)\|^2_{X^s}.
\end{align*}
Directly from \eqref{eq:systhypgeneral}, we have:
\begin{align*}
  \|\d_t \bu^\eps(t)\|_{L^\infty} &\le C \(
  \|\bu^\eps(t)\|_{L^\infty}\) \|\d_x \bu^\eps(t)\|_{L^\infty} +
  \|\d_{xx}^2 a^\eps(t)\|_{L^\infty}\\
& \le C \(
  \|\bu^\eps(t)\|_{X^s}\) \| \bu^\eps(t)\|_{X^s}.
\end{align*}
Since $SL$ is skew-symmetric, we have
\begin{equation*}
  \RE\< SL P \bu^\eps(t),P\bu^\eps(t)\>=0,
\end{equation*}
which prevents any loss of regularity in the estimates. For the
quasi-linear term involving the matrix $A$, we note that since
$SA$ is symmetric, commutator estimates (see \cite{LannesJFA}) yield:
\begin{align*}
   \< S  P
  \( A(\bu^\eps)\partial_x \bu^\eps\),P\bu^\eps(t)\> &\le C\( \|
  \bu^\eps(t) \|_{L^\infty}\) \|P\bu^\eps(t)\|_{L^2}^2
  \|\d_x\bu^\eps(t)\|_{L^\infty}\\
&\le C\( \|
  \bu^\eps(t) \|_{X^s }\)\|P\bu^\eps(t)\|_{L^2}^2.
\end{align*}
Finally, we have:
\begin{equation*}
  \frac{d}{dt} \<SP\bu^\eps(t),P\bu^\eps(t)\>\le C\(
  \|\bu^\eps(t)\|_{X^s}\) \|\bu^\eps(t)\|_{X^s}^2.
\end{equation*}
This estimate, along with the $L^\infty$-estimate, shows that on a
sufficiently small time interval $[0,T]$, with $T>0$ independent
of $\eps \in [0,1]$, \eqref{eq:solong} holds, hence the existence of a
unique solution. 
The fact that the local existence time does not depend on $s>5/2$
follows from the continuation principle based on Moser's calculus and
tame estimates (see e.g. \cite[Section~2.2]{Majda} or
\cite[Section~16.1]{Taylor3}).
\end{proof}

\begin{corollary}
  Under the assumptions of Proposition~\ref{prop:existence}, if we suppose
  in addition that $\Phi_0^\prime\in L^2(\R)$, then
  \eqref{Schrod1} has a unique solution $u^\eps \in
  C([0,T];X^s(\R))$, where $T$ is given by Proposition~\ref{prop:existence}.
\end{corollary}
\begin{proof}
From  Proposition~\ref{prop:existence}, \eqref{av-eps}
has a solution $(v^\eps,a^\eps)\in C([0,T];X^s\times X^s)$. Plugging
this information into \eqref{av-eps}, we infer
\begin{align*}
  \|v^\eps(t)\|_{L^2}&\le \|\Phi_0^\prime\|_{L^2} + \int_0^t
  \|v^\eps(\tau)\|_{L^\infty}\|\d_x v^\eps(\tau)\|_{L^2}d\tau \\
&\quad + C
  \int_0^t \left\|\frac{1}{a^\eps(\tau)}\right\|_{L^\infty}\|\d_x a^\eps(\tau)\|_{L^2}d\tau.
\end{align*} 
 Therefore, $v^\eps\in C([0,T];L^2)$, and  $\Phi^\eps$, stemming from
$v^\eps$ \emph{via} the formula \eqref{eq:v2phi}, satisfies $\Phi^\eps
\in C([0,T];X^{s+1})$. 
  The existence part follows readily, since $X^s(\R)$ is an algebra, by
  setting $u^\eps = a^\eps e^{i\Phi^\eps/\eps}$. 

For the uniqueness property, consider two such solutions $u^\eps,\tilde
u^\eps\in C([0,T];X^s)$, and set $w^\eps= u^\eps-\tilde u^\eps$. It
satisfies
\begin{equation}\label{eq:w}
  i\eps\d_t w^\eps +\frac{\eps^2}{2}\d_{xx}^2w^\eps = \lambda\(
    \ln(|u^\eps|^2)u^\eps -   \ln(|\tilde u^\eps|^2)\tilde
    u^\eps\),\quad w^\eps_{\mid t=0}=0. 
\end{equation}
  Recall the pointwise estimate from \cite{CaHa80} (see also
  \cite[Lemma~9.3.5]{CazCourant}),
  \begin{equation*}
    \left|\IM \(
    \ln(|u^\eps|^2)u^\eps -   \ln(|\tilde u^\eps|^2)\tilde
    u^\eps\)(u^\eps-\tilde u^\eps)\right|\le 4 |u^\eps-\tilde
  u^\eps|^2. 
  \end{equation*}
Multiply \eqref{eq:w} by $\overline{w^\eps}$, integrate on an interval
$I=[M_-,M_+]$, and take the imaginary part. This yields, along with
the above estimate,
\begin{equation*}
  \frac{\eps}{2}\frac{d}{dt}\int_I |w^\eps(t,x)|^2dx +
  \frac{\eps^2}{2}\IM\int_I \overline{w^\eps}\d_{xx}^2 w^\eps \le
  4\lambda \int_I |w^\eps(t,x)|^2dx.
\end{equation*}
We have, by integration by parts,
\begin{equation*}
  \IM\int_I \overline{w^\eps}\d_{xx}^2 w^\eps =
  \IM\overline{w^\eps}(t,M_+)\d_x w^\eps(t,M_+)-
  \IM\overline{w^\eps}(t,M_-)\d_x w^\eps(t,M_-). 
\end{equation*}
Since $w^\eps \in C([0,T];X^1)$, we can choose sequences $M_\pm^n\to
\pm\infty$ along which the above term goes to zero, and the Gronwall lemma
implies $\|w^\eps(t)\|_{L^2}\equiv 0$.
\end{proof}

\subsection{Asymptotic expansion}
\label{sec:asymptotic-expansion}

Proposition~\ref{prop:existence} with $\eps=0$ yields the existence of
a unique solution $(v,a)\in C([0,T];(X^s(\R ))^2)$ to \eqref{eq:limit}
As a direct consequence of Proposition~\ref{prop:existence} and
\cite[Proposition~3.1]{ACIHP}, we have:
\begin{proposition}\label{prop:DA1}
  Under the assumptions of Proposition~\ref{prop:existence}, there
  exists $C$ independent of $\eps\in [0,1]$ such that
  \begin{equation*}
    \|\d_x \(\Phi^\eps-\Phi\)\|_{L^\infty([0,T];X^{s-2})}+
    \|a^\eps-a\|_{L^\infty([0,T];X^{s-2})}\le C\eps. 
  \end{equation*}
\end{proposition}

\subsection{Convergence of the Wigner transform}
\label{sec:wigner-bkw}
For $u^\eps = a^\eps e^{i\Phi ^\eps /\eps }$, with $(\Phi ^\eps ,
a^\eps )$ solution to \eqref{Phi-eps}--\eqref{a-eps}, or equivalently
$(\partial _x\Phi ^\eps , a^\eps )$ solution to
\eqref{av-eps},  the Wigner transform defined in \eqref{df1-wig} is equal to
\begin{equation}\label{eq:wigner-bkw}
W^\eps (t,x,\xi )= \int e^{i\xi v}a^\eps\(t,x-\frac{\eps
}{2}y\)\overline{a^\eps }\(t,x+\frac{\eps }{2}y\)e^{i\varphi^\eps(t,x,y)/\eps}dy,
\end{equation}
where
\begin{equation*}
  \varphi^\eps(t,x,y) = \Phi ^\eps
    \(t,x-\frac{\eps }{2}y\)-\Phi ^\eps \(t,x+\frac{\eps }{2}y\).
\end{equation*}
\begin{theorem}\label{wig}
When $\eps \to 0$, the Wigner transform $W^\eps $ of $u^\eps $ weakly
converges  to the bounded measure  
\begin{equation*}
\mu (t,dx,d\xi)= \lvert a(t,x)\rvert ^2dx\otimes \delta _{\xi= \partial _x\Phi (t,x)},
\end{equation*}
where $(\d _x\Phi ,a)$ is a solution of \eqref{eq:limit}.
Moreover, $\mu $ is a solution to \eqref{vla1}.
\end{theorem}
\begin{proof}
In view of Proposition~\ref{prop:DA1}, $a^\eps = a+ r_a^\eps$ and
$\d_x \Phi^\eps = \d_x \Phi +  r_v^\eps$, with 
\begin{equation*}
  \|r_a^\eps\|_{L^\infty([0,T];X^{s-2})}+\|r_v^\eps\|_{L^\infty([0,T];X^{s-2})}
  \le C\eps. 
\end{equation*}
Therefore,
\begin{equation*}
W^\eps (t,x,\xi )= \int e^{iy\( \xi -\partial _x\Phi (t,x)\)
}a\(t,x-\frac{\eps }{2}y\)\bar{a}\(t,x+\frac{\eps }{2}y\)dy+R^\eps
_1+R^\eps _2+R^\eps _3, 
\end{equation*}
where
\begin{align*}
&R^\eps _j (t,x,\xi )= \int e^{iy\( \xi -\partial _x\Phi (t,x)\)
  }r^\eps _j(t,x,y)dy,\quad 1\le j\le 3,\\ 
&r^\eps _1(t,x,y)= a^\eps\(t,x-\frac{\eps }{2}y\)\overline{a^\eps}\(t,x+\frac{\eps
  }{2}y\)\( e^{i\(\varphi^\eps(t,x,y) +\eps\d_x \Phi(t,x)\)/\eps}-1\),\\ 
&r^\eps _2(t,x,y)= \bar{a}\(t,x+\frac{\eps }{2}y\)r_a^\eps \(t,x-\frac{\eps
  }{2}y\)+a\(t,x-\frac{\eps }{2}y\)\overline{r_a^\eps }\(t,x+\frac{\eps
  }{2}y\),\\ 
&r^\eps _3(t,x,y)= r_a^\eps \(t,x-\frac{\eps }{2}y\)\overline{r_a^\eps
  }\(t,x+\frac{\eps }{2}y\). 
\end{align*}
 Propositions~\ref{prop:existence} and \ref{prop:DA1} yield, along with
 Taylor's formula for the term $r_1^\eps$,
 \begin{equation*}
   \|r_j^\eps\|_{L^\infty([0,T]\times\R^2)}\le C\eps, \quad 1\le j\le 3.
 \end{equation*}
Consequently, $W^\eps $ tends to $\lvert a\rvert ^2dx\otimes \delta
_{\xi= \partial _x\Phi }$ in $\mathcal{M}_b([ 0,T] \times \R ^2)$
when $\eps $ tends to zero. \\ 
Moreover, denote by $\(\cdot ,\cdot \)$ the duality between bounded
measures on $[ 0,T] \times \R ^2$ and continuous functions with
compact support in $[ 0,T] \times \R ^2$. 
For any test function $\alpha (t,x,\xi)\in C^1$ with compact support
in $[ 0,T] \times \R ^2$, it holds  
\[ \begin{aligned}
&\(\mu ,\partial _t\alpha +\xi \partial _x\alpha -\lambda
\( \partial _x \ln\lvert a\rvert ^2\) \partial _\xi \alpha \)\\ 
&= \int \lvert a(t,x)\rvert ^2\( \partial _t\alpha (t,x,\partial
_x\Phi (t,x))+\partial _x\Phi (t,x)\partial _x\alpha (t,x,\partial
_x\Phi (t,x))\)dxdt\\
&\quad -\lambda \int \lvert a(t,x)\rvert ^2\(  \( \partial _x \ln\lvert
a\rvert ^2\) \partial _\xi \alpha (t,x,\partial _x\Phi (t,x))\) dxdt\\ 
&= \int \lvert a\rvert ^2\( \partial _t\( \alpha (t,x,v (t,x))\)
-\d_T v \partial _\xi \alpha 
(t,x,v (t,x))\)dxdt\\
&+\int \lvert a\rvert ^2\(\partial _x\( v\,
\alpha (t,x,v (t,x))\) -\d_x v\,\alpha
(t,x,v )-v \d_x v\, \partial_\xi
\alpha (t,x,v   )\)dxdt\\
&\quad -\lambda \int \lvert a\rvert ^2\(  \( \partial _x \ln\lvert
a\rvert ^2\) \partial _\xi \alpha (t,x,v (t,x))\) dxdt\\ 
&= -\int \alpha (t,x,v )\( \partial _t\lvert a\rvert
^2+v \partial _x\lvert a\rvert ^2+\lvert
a\rvert ^2\d_x v \) dxdt\\ 
&\quad-\int \lvert a\rvert ^2\partial _\xi \alpha (t,x,v
)\( \d_t v +v  \d_xv 
+\lambda \partial _x \ln\lvert a\rvert ^2\) dxdt.
\end{aligned}\]
This is zero, in view of \eqref{eq:limit}, since
\begin{align*}
  \partial _t\lvert a\rvert ^2+\partial _x\Phi\partial _x\lvert
  a\rvert ^2 +|a|^2\partial ^2_{xx}\Phi  &= 2\RE \overline a \( \d_t a
  +\d_x \Phi \d_x a +\frac{a}{2}\d_{xx}^2\Phi \)=0.
\end{align*}
Therefore, any solution to \eqref{eq:limit} yields a solution to
\eqref{vla1}.
\end{proof}

\section{The bound from below of the density}
\label{sec:6}

It remains to prove that the density is bounded from below to complete the proof of Theorem \ref{th2}. The result from \cite{Ch-p} yields:
\begin{proposition}\label{prop:chen}
  Under the assumptions of Theorem~\ref{th2}, the density $\rho$
  solution to \eqref{isen1} satisfies
    \begin{equation*}
  \rho(t)\ge \frac{\rho_{0*}}{1+Ct},\quad t\in [ 0,+\infty [ ,
 \end{equation*}
for some constant $C>0$. 
\end{proposition}
\begin{proof}
In   \cite{Ch-p}, the proof is presented for the pressure law
$p(\rho)=\rho$ in \eqref{isen1} 
replaced by the isentropic one, $p(\rho)=\rho^\gamma$, with
$\gamma>1$. We simply perform the standard modification to adapt the
proof to the isothermal case (see e.g. \cite{Chen,Da00}). Note that
Proposition~\ref{prop:chen} is also a consequence of
\cite[Remark~3.3]{Ch-p}: we sketch the proof in the isothermal case for $\lambda = 1$ 
for the convenience of the reader. 
\smallbreak

First, the isothermal Euler equation in Lagrangian coordinates reads
(see e.g. \cite{Da00})
\begin{equation}\label{eq:lagrange}
  \tau_t - u_x=0,\quad u_t -\frac{\tau_x}{\tau^2}=0,
\end{equation}
with $\tau=1/\rho$. Setting
\begin{equation*}
  s= u-\ln \tau,\quad r = u+\ln \tau,
\end{equation*}
we find
\begin{equation*}
  \d_+s=\d_-r =0,
\end{equation*}
where
\begin{equation*}
  \d_\pm = \d_t \pm\frac{1}{\tau}\d_x.
\end{equation*}
Following \cite{Ch-p}, introduce
\begin{equation*}
  \alpha=s_x,\quad \beta=r_x.
\end{equation*}
By assumption, we know that $u$, $\tau$, $\alpha$ and $\beta$ are
uniformly bounded at time $t=0$: there exists $M>0$ such that
\begin{equation*}
  \sup_{x\in \R} (\alpha(0,x),\beta(0,x))<M.
\end{equation*}
We check
\begin{align}
 \label{eq:alpha}     & \d_+ \alpha = \frac{1}{2\tau}\alpha(\beta-\alpha),\\
\label{eq:beta}
& \d_-\beta = \frac{1}{2\tau}\beta(\alpha-\beta).
\end{align}
As in \cite{Ch-p}, 
\begin{equation}\label{eq:2.7}
  \sup_{(t,x)\in [0,T)\times\R} (\alpha(t,x),\beta(t,x))<M,
\end{equation}
where $T>0$ is such that \eqref{eq:lagrange} has a $C^1$ solution on
$[0,T)$.  \eqref{eq:2.7} is established by contradiction: assume for
instance that $\alpha(t_0,x_0)=M$. Because the wave speed
$\frac{1}{\tau}$ is bounded from above, we can consider the
characteristic triangle with vertex $(t_0,x_0)$ and lower boundary at
time $t=0$, denoted by $\Omega$. Let $t_1 (\le t_0)$ be the first time such that
$\alpha(t_1,x_1)=M$ or $\beta(t_1,x_1)=M$ in $\Omega$. Assume
$\alpha(t_1,x_1)=M$ for instance (the other case is similar), and let
$\Omega_1$ denote the 
characteristic triangle with vertex $(t_1,x_1)$:
\begin{equation*}
  \sup_{(t,x)\in \Omega_1,\ t<t_1} (\alpha(t,x),\beta(t,x))<M.
\end{equation*}
Since $\alpha(t_1,x_1)=M$, there exists $t_2\in [0,t_1)$ such that
\begin{equation*}
  \inf_{(t,x)\in \Omega_1,\ t\ge t_2}\alpha(t,x)>0.
\end{equation*}
Let $t_2<t<t_1$. On $[t_2,t]$, \eqref{eq:alpha} yields
\begin{equation*}
  \d_+\alpha \le K (M\alpha -\alpha^2).
\end{equation*}
Integrating along some charateristic between $t_2$
and $t$, we find
\begin{equation*}
  \frac{1}{M}\ln\frac{\alpha(t)}{M-\alpha(t)}\le
  \frac{1}{M}\ln\frac{\alpha(t_2)}{M-\alpha(t_2)}  + K(t-t_2).
\end{equation*}
Letting $t\to t_1$ then leads to a contradiction. 

We infer that $u_x$ is bounded from above, hence $\tau_t$ is bounded from above, in view of
\eqref{eq:lagrange}: $\tau$ grows at most linearly in time, and since
$\rho=1/\tau$, the proposition follows. 

\end{proof}

\bibliographystyle{siam}

\bibliography{biblio}
\end{document}